\documentclass[12pt,a4paper,leqno]{amsart}

\usepackage{amsmath,amsthm,amscd,amsfonts,amssymb,graphicx,wasysym,soul,enumitem}
\usepackage{caption}
\usepackage{subcaption}

\usepackage[right=25mm, left=25mm, top=25mm, bottom=25mm]{geometry}

\usepackage[dvipsnames]{xcolor}

\sethlcolor{yellow}

\newtheorem{theorem}{Theorem}[section]

\theoremstyle{remark}

\newtheorem{remark}{Remark}[section]

\theoremstyle{definition}

\numberwithin{equation}{section}

\fussy

\begin{document}
\title  {Mittag-Leffler Poisson Distribution Series and Their Application to Univalent Functions}	
\maketitle

\author{
    \textbf{K. Marimuthu$^{\textbf{1}}$}, 
    \textbf{A. Jeeva$^{\textbf{2}}$}, 
    \textbf{Nasir Ali$^{\textbf{3,*}}$}, \\
$ ^{\textbf{1,2}} $ Department of Mathematics, Vel Tech Rangarajan Dr. Sukunthala R\&D  Institute of Science and Technology, Avadi, Chennai-600062, Tamilnadu, India.\\ 
$ ^{\textbf{3}} $Department of Mathematics,COMSATS University Islamabad, Vehari Campus, Pakistan\\

{Email:marimuthu.k@velhightech.com (K. Marimuthu), drjeevaa@veltech.edu.in (A. Jeeva), nasirzawar@gmail.com (Nasir Ali)}\\
\begin{abstract}
		In this study, we establish a significant connection between certain subclasses of complex order univalent functions and the Mittag-Leffler-type Poisson distribution series. We provide criteria for these series to belong to the specific subclasses. The primary goal of this investigation is to derive necessary and sufficient conditions for the Mittag-Leffler-type Poisson distribution series $\mathcal{P}(p,u,v)(z)$ to be included in the classes $\mathcal{S}(\delta,\eta,\tau)$ and $\mathcal{R}(\delta,\eta,\tau)$. These findings enhance our understanding of the structural properties of univalent functions and extend the applicability of Mittag-Leffler-type distributions in complex analysis.
\vskip5pt
	
	\end{abstract}

\maketitle

%%%%%%%%%%%%%%%%%%%%%%%%%%%%%%%%%%%%%%%%%%%%%%%%%%%%%%%%%%%%%%%%%%%%%****************************************

\section{Introduction}
Univalent functions, or functions that are holomorphic and injective, play a pivotal role in complex analysis due to their rich structural properties and extensive applications. These functions, particularly when defined on the unit disk, have been studied extensively to understand their geometric behavior and functional relationships. One of the intriguing areas of research involves the characterization of subclasses of univalent functions through various mathematical constructs and series.

The Mittag-Leffler function, a generalization of the exponential function, and its associated distributions have garnered significant attention for their applications in diverse fields such as probability, statistics, and physics. In this paper, we delve into the Mittag-Leffler-type Poisson distribution series and explore its relationship with specific subclasses of univalent functions of complex order. By establishing new criteria and conditions, we aim to enhance the understanding of how these series interact with and define univalent function subclasses.
The novelty of this research lies in the application of the Mittag-Leffler-type Poisson distribution series to the domain of univalent functions. While both the Mittag-Leffler function and univalent functions have been studied extensively in isolation, their combined exploration within the context of complex order functions introduces a new dimension to the field. This work not only bridges the gap between these two areas of study but also provides fresh insights and criteria for the inclusion of Poisson distribution series in univalent function subclasses. Such an integrative approach has the potential to open up new avenues for research and application in complex analysis and related disciplines.

\textbf{Problem Statement: }The primary problem addressed in this paper is to determine the necessary and sufficient conditions for the Mittag-Leffler-type Poisson distribution series $\mathcal{P}(p,u,v)(z)$ to belong to the subclasses $\mathcal{S}(\delta,\eta,\tau)$ and $\mathcal{R}(\delta,\eta,\tau)$ of univalent functions. Despite the individual significance of Mittag-Leffler functions and univalent function theory, a comprehensive framework that connects these series with specific univalent function subclasses remains underexplored. This research seeks to fill this gap by providing a detailed analysis and establishing concrete criteria that can serve as a foundation for future studies and applications.

The class of analytic functions denoted by  $\mathcal{A}$ is defined as

\begin{equation}\label{11}
	f(z)=z+\sum\limits_{s=2}^{\infty} a_{s}z^{s}, 
\end{equation}
which are analytic  in the open unit disc $\mathbb{E}=\left\{z: z\in \mathbb{C} \text{ such  that } 0<|z|<1 \right\}$. 

Let us consider the subclass $\mathcal{H}$ of $\mathcal{A}$  consists of functions of the form,
\begin{equation}
	f(z)=z-\sum\limits_{s=2}^{\infty} a_{s}z^{s}, a_{s}\geq 0.
\end{equation}

We give a following condition for	a function $f\in \mathcal{A}$  to be starlike of complex order $\delta(\delta\in \mathbb{C}^{\ast})$ as 
\begin{equation}\label{13}
	\mathcal{R} \left\{1+\frac{1}{\delta}\left(\frac{zf^{\prime}(z)}{f(z)}-1\right)\right\}>0, z\in \mathbb{E}.
\end{equation}
which is both necessary and sufficient  when $\frac{f(z)}{z}\neq0$.

The class of all those functions is denoted by $\mathcal{S}(\delta)$. The class $\mathcal{S}(\delta)$ was  proposed by Nasr and Aouf \cite{Nasr 1985}.
% In the above equation \eqref{13}, when the values of $\delta$ are assumed to be 1 and $e^{-i\theta}cos\theta (|\theta|<\frac{\pi}{2}$ and $\theta$ is real value), the famous classes of functions namely starlike $(\mathcal{S}^{*})$ and spirallike $(\mathcal{S}^{\theta})$ are obtained. For  $\delta=1-\zeta (0\leq\zeta<1)$ and $\delta=(1-\zeta)e^{-i\theta}cos\theta (|\theta|<\frac{\pi}{2}$ and $\theta$ is real value), we obtain the same classes of order  $\zeta$ denoted by $\mathcal{S}^{*}(\zeta)$ and $\mathcal{S}^{\theta}_{\zeta}$. 
\begin{remark}
	In the above inequality \eqref{13}, for different values of $\delta$, we obtain the famous class of functions as follows:
	\begin{enumerate}[label=(\roman*)]
		\item  For $\delta=1$, we have the class of starlike function $(\mathcal{S}^{*})$;
		
		\item  For $\delta=e^{-i\theta}cos\theta, \  \ (|\theta|<\frac{\pi}{2})$, we have the class of spirallike function $(\mathcal{S}^{\theta})$;
		
		\item  For $\delta=1-\zeta, \  \ (0\leq\zeta<1)$, we have the class of starlike function $(\mathcal{S}^{\ast}_{\zeta})$ of order $\zeta$;
		
		\item  For $\delta=(1-\zeta)e^{-i\theta}cos\theta, \  \ (|\theta|<\frac{\pi}{2})$, we have the class of spirallike function of order  $\zeta$ $(\mathcal{S}^{\theta}_{\zeta})$;
	\end{enumerate}
\end{remark}

The following necessary and sufficient condition gives the function $f\in \mathcal{A}$ to be convex of complex order $\delta(\delta\in \mathbb{C}^{\ast})$ when  $f^{\prime}(z)\neq0$, it is given by
\begin{equation}\label{14}
	\mathcal{R} \left\{1+\frac{1}{\delta}\left(\frac{zf^{\prime\prime}(z)}{f^{\prime}(z)}\right)\right\}>0, z\in \mathbb{E}.
\end{equation}

The class $\mathcal{C}(\delta)$ was introduced by Wiatrowski \cite{W 1971}. It is worth noting that $f\in\mathcal{C}(\delta)$ if and only if $zf^{\prime}\in\mathcal{S}(\delta)$. 
%Based on the estimate of  $\delta$, the above inequality(1.3) gives  the subclasses of  convex univalent functions denoted by $\mathcal{C}$ and $\mathcal{C}(\zeta)$ when  $\delta=1$ and $\delta=1-\zeta(0\leq\zeta<1)$  also the class of $\theta$-Robertson functions $\mathcal{S}^{\theta}(\delta=e^{-i\theta}cos\theta(|\theta|<\frac{\pi}{2})$ and $\theta$ is real value).
\begin{remark}
	In the above inequality \eqref{14}, for different values of $\delta$, we obtain the famous class of functions as follows:
	\begin{enumerate}[label=(\roman*)]
		\item For $\delta=1$, we have the class of convex function $\mathcal{C}$;
		
		\item  For $\delta=1-\zeta, \  \ (0\leq\zeta<1)$, we have the class of convex function  $(\mathcal{S}_{\zeta})$ of order $\zeta$;
		
		\item  For $\delta=e^{-i\theta}cos\theta, \  \ (|\theta|<\frac{\pi}{2})$, we have the class of $\theta$-Robertson function $(\mathcal{S}^{\theta})$;
		
		%	(iv) For $\delta=(1-\zeta)e^{-i\theta}cos\theta (|\theta|<\frac{\pi}{2})$, we have the class of spirallike function of order  $\zeta$ $(\mathcal{S}^{\theta}_{\zeta})$.
	\end{enumerate}
\end{remark}

The function $f\in \mathcal{A}$ is known as close to convex function of order $\delta(\delta\in \mathbb{C}^{\ast})$ provided with the following necessary and sufficient condition 

\begin{equation*}
	\mathcal{R} \left\{1+\frac{1}{\delta}(f^{\prime}(z)-1)\right\}, z\in \mathbb{E}.
\end{equation*}

The class $\mathcal{R}(\delta)$ was introduced by Owa \cite{Owa 1988} and Addul Halim \cite{Abdul Halim 1999}.

The following subclasses of $\mathcal{A}(s)$ were obtained by Altintas et al. \cite{Altintas 2000}, and they consist of functions of the form 
\begin{equation*}
	f(z)=z-\sum\limits_{s=j+1}^{\infty}a_{s}z^{s}
\end{equation*}
\textbf{Definition 1.1.} \cite{Altintas 2000} Let $\mathcal{S}_{s}(\delta,\eta,\tau)$ indicate the subclass of $\mathcal{A}(s)$ containing of function	$f$ that holds the inequality
\begin{equation*}
	\left|\frac{1}{\delta}\left(\frac{zf^{\prime}(z)+\eta z^{2}f^{\prime\prime}(z)}{\eta zf^{\prime}(z)+(1-\eta)f(z)}-1\right)\right|< \tau,
\end{equation*}
where $z\in \mathbb{E}, \delta\in \mathbb{C}^{\ast}, 0<\tau \leq 1$\  \ and \  \ $0\leq \eta \leq 1.$

Also let  $\mathcal{R}_{s}(\delta,\eta,\tau)$ indicate the subclass of $\mathcal{A}(s)$ consisting of functions $f$ that holds the inequality	
\begin{equation*}
	\left|\frac{1}{\delta}(f^{\prime}(z)+\eta zf^{\prime\prime}(z)-1)\right|< \tau. 
	%(z\in E, \delta\in C^{*}, 0<\tau \leq 1, 0\leq \eta \leq 1).
\end{equation*}

When on the assumption of functions in the above classes, the authors found the following coefficient inequalities in their paper: \\
\textbf{Lemma 1.2.} \cite{Altintas 2000} Let the function $f\in \mathcal{A}(s)$, then $f \in \mathcal{S}_{s}(\delta,\eta,\tau)$ if and only if
\begin{equation*}
	\sum\limits_{s=j+1}^{\infty}[\eta(s-1)+1][s+\tau|\delta|-1]a_{s}\leq\tau|\delta|
\end{equation*}
\textbf{Lemma 1.3.} \cite{Altintas 2000} Let the function $f\in \mathcal{A}(s)$, then $f \in \mathcal{R}_{s}(\delta,\eta,\tau)$ if and only if
\begin{equation*}
	\sum\limits_{s=j+1}^{\infty}s[\eta(s-1)+1]a_{s}\leq\tau|\delta|
\end{equation*}

If the functions $f$ belongs to the classes  $\mathcal{S}_{s}(\delta,\eta,\tau)$ and  $\mathcal{R}_{s}(\delta,\eta,\tau)$, we assume that $ s =1$, then we write  $\mathcal{S}_{1}(\delta,\eta,\tau)= \mathcal{S}(\delta,\eta,\tau)$ and  $\mathcal{R}_{1}(\delta,\eta,\tau)= \mathcal{R}(\delta,\eta,\tau)$.\\
\textbf{Definition 1.4.} Let $\mathcal{K}(\eta,\delta)$ indicate the subclass of $\mathcal{A}$ containing functions of the form \eqref{11} which holds
\begin{equation*}
	\mathcal{R} \left\{1+\frac{1}{\delta}\left(\frac{zf^{\prime}(z)+\eta z^{2}f^{\prime\prime}(z)}{\eta zf^{\prime}(z)+(1-\eta)f(z)}-1\right)\right\}> 0,
\end{equation*}
where	$z\in \mathbb{E}, \delta\in \mathbb{C}^{\ast}\  \ \text{and}\  \  0\leq \eta \leq 1.$

Also, let $\mathcal{R}(\eta,\delta)$ indicate the subclass of $\mathcal{A}$ containing functions of the form \eqref{11} which holds
\begin{equation*}
	\mathcal{R}\left\{1+\frac{1}{\delta}(f^{\prime}(z)+\eta zf^{\prime\prime}(z)-1)\right\}> 0. %(z\in E, \delta\in C^{*},  0\leq \eta \leq 1).
\end{equation*}

Altıntas¸ et al.\cite{Altintas 1999} and Aouf \cite{Aouf 2013} investigated and studied the classes $\mathcal{K}(\eta,\delta)$ and $\mathcal{R}(\eta,\delta)$.  The following lemmas have been proved by Aouf \cite{Aouf 2013} and given with sufficient conditions for the functions $f$ to be in the classes $\mathcal{K}(\eta,\delta)$ and $\mathcal{R}(\eta,\delta)$.\\
\textbf{Lemma 1.5.}\cite{Aouf 2013} Assume that the function $f$ defined in \eqref{11} satisfies the following inequality:
\begin{equation*}
	\sum\limits_{s=2}^{\infty}[\eta(s-1)+1][(s-1)+|2\delta+s-1|]|a_{s}|\leq2|\delta|,
\end{equation*}
then $f\in \mathcal{K}(\eta,\delta)$.\\
\textbf{Lemma 1.6.} \cite{Aouf 2013} Let $f\in A$ then $f\in \mathcal{R}(\eta,\delta)$ if it holds the inequality,
\begin{equation*}
	\sum\limits_{s=2}^{\infty}s[\eta(s-1)+1]|a_{s}|\leq2|\delta|.
\end{equation*}

We can see that  $\mathcal{K}(0,\delta)=\mathcal{S}(\delta), \mathcal{K}(1,\delta)=\mathcal{C}(\delta)$ and  $\mathcal{R}(0,\delta)=\mathcal{R}(\delta)$.

In geometric function theory, the series form of all possible ditributions say binomial, poisson, pascal and mittag-leffler-type poisson has been introduced and associated with many subclasses of univalent functions by many authors and it has played an important role in recent development in the area of complex analysis. Bansal and Prajapat \cite{D.Bansal} have examined geometric features of the Mittag-Leffler function $E_{\alpha,\beta}(z)$, including starlikeness, convexity, and close-to-convexity. In \cite{Raducanu}, results of differential subordination connected with the generalised Mittag-Leffler function were also found.

Recently, Porwal and Dixit \cite{Porwal 2017} introduced the idea of   mittag–leffler-type poisson distribution and obtained the moment generating function. Bajpai \cite{Bajpai} introduced the same and obtained some necessary and sufficient conditions.

The PMF(probability mass function) of the Mittag–Leffler-type Poisson distribution is given by
\begin{equation*}
	\mathcal{H}(p,u,v;m)=\frac{p^{n}}{E_{u,v}(p)\Gamma(u n+v)},\  \ n\geq0,
\end{equation*}

where $E_{u,v}(z)=\sum\limits_{s=0}^{\infty}\frac{z^{s}}{\Gamma(u n+v)}$, $u,v,z \in \mathbb{C}$. It can be seen that
the series $E_{u,v}(z)$ converges for all finite values of z if $\mathcal{R}(u)>0,\  \ \mathcal{R}(v)>0$.\\

The Mittag–Leffler-type Poisson distribution series is given as\\
\begin{equation}
	\mathcal{F}(p,u,v)(z)=z+\sum\limits_{s=2}^{\infty}\frac{p^{s-1}}{E_{u,v}(p)\Gamma(u(s-1)+v)}z^{s}.
\end{equation}

Here, we define the following series and give the necessary and sufficient
condition for this series to be in the some subclasses of univalent functions.

\begin{equation}
	\mathcal{P}(p,u,v)(z)=2z-\mathcal{F}(p,u,v)(z)=z-\sum\limits_{s=2}^{\infty}\frac{p^{s-1}}{E_{u,v}(p)\Gamma(u(s-1)+v)}z^{s}.
\end{equation}
%	Now, we considered the linear operator
%	\begin{equation}
	%		\mathcal{G}_{q}^{m}(z): \mathcal{A}\rightarrow \mathcal{A}
	%	\end{equation}
%defined by the convolution or Hadamard product
%	\begin{equation}
	%	\mathcal{G}_{q}^{m}(z)=\mathcal{P}_{q}^{m}(z)*f(z)=z+\sum\limits_{s=2}^{\infty}
	%	\begin{pmatrix}
		%	s+m-2 \\ m-1
		%	\end{pmatrix}
	%	q^{s-1}(1-q)^{m}a_{s}z^{s}, z\in E
	%	\end{equation}

Motivated by the works of Porwal.et.al.\cite{Porwal 2020}, Murugusundramoorthy et.al. \cite{G.M 2020} (see also \cite{Alessa}, \cite{M.G.Khan}, \cite{V 2021}, \cite{Yalcin 2024}, \cite{Mari 2023} ), we obtain the necessary and sufficient conditions for Mittag–Leffler-type Poisson distribution series $\mathcal{P}(p,u,v)(z)$ to be in the classes     $\mathcal{S}(\delta,\eta,\tau)$ and $\mathcal{R}(\delta,\eta,\tau)$, and the  series $\mathcal{F}(p,u,v)(z)$ to be in the classes $\mathcal{K}(\eta,\delta)$ and $\mathcal{R}(\eta,\delta)$ these must be sufficient conditions.

\section{\textbf{Main Results}}
In this section, we obtain the necessary and sufficient conditions for $\mathcal{P}(p,u,v)(z) \in \mathcal{S}(\delta,\eta,\tau)$.\\
\begin{theorem}
	Let $p,u \geq 1, v>2$. The function $\mathcal{P}(p,u,v)(z) \in \mathcal{S}(\delta,\eta,\tau)$ if and only if
	\begin{equation}
		\begin{array}{rcl}
			\frac{1}{u^{2}E_{u,v}(p)} \Big[\eta \left( E_{u,v-2}(p)-\frac{1}{\Gamma (v-2)} \right)+\Big((3-2v)\eta+u(1+\eta\tau|\delta|)\Big)\Big(E_{u,v-1}(p)-\frac{1}{\Gamma (v-1)}\Big)\\
			
			+\Big(\eta(1-v)^{2}+u(1+\eta\tau|\delta|)(1-v)+u^{2}\tau|\delta|\Big) \Big(E_{u,v}(p)-\frac{1}{\Gamma v}\Big) \Big] \leq \tau |\delta|. 
		\end{array}
	\end{equation} 
\end{theorem}
\begin{proof} Since
	\begin{equation*}
		\mathcal{P}(p,u,v)(z)=z-\sum\limits_{s=2}^{\infty}\frac{p^{s-1}}{E_{u,v}(p)\Gamma(u(s-1)+v)}z^{s}
	\end{equation*}
	In light of lemma 1.2, it is sufficient to demonstrate that
	\begin{equation}
		\sum\limits_{s=2}^{\infty}[\eta(s-1)+1][s+v|\delta|-1]
		\frac{p^{s-1}}{E_{u,v}(p)\Gamma(u(s-1)+v)}\leq\tau|\delta|
	\end{equation}
	Now, we have
	\begin{equation*}
		\begin{array}{rcl}
			&&\sum\limits_{s=2}^{\infty}[\eta(s-1)+1][s+v|\delta|-1]
			\frac{p^{s-1}}{E_{u,v}(p)\Gamma(u(s-1)+v)} \\
			
			&=&\sum\limits_{s=1}^{\infty}\Big[(\eta s+1)(s+\tau|\delta|)\Big] \frac{p^{s}}{E_{u,v}(p)\Gamma(u(s)+v)}\\
			%\end{array}
			%\end{equation*}
			%\begin{equation*}
			%\begin{array}{rcl}
			&=&\frac{1}{E_{u,v}(p)}\sum\limits_{s=1}^{\infty}\Big[\eta s^{2}+s(\eta\tau|\delta|+1)+\tau|\delta|\Big]\frac{p^{s}}{\Gamma(u s+v)}\\
			
			&=&\frac{1}{E_{u,v}(p)}\sum\limits_{s=1}^{\infty}\Big[\frac{\eta}{u ^{2}}\Big[(u s+v-1)(u s+v-2)+(3-2v)(u s+v-1)+(1-v)^{2}\Big]\\
			
			&&+\frac{(1+\eta\tau|\delta|)}{u}\Big[(u s+v-1)+(1-v)\Big]+ \tau|\delta| \Big]\frac{p^{s}}{\Gamma (u s+v)}\\
			
			&=&\frac{1}{E_{u,v}(p)} \Big[ \frac{\eta}{u^{2}}\sum\limits_{s=1}^{\infty}\frac{(u s+v-1)(u s+v-2)p^{s}}{(u s+v-1)!}+\frac{\eta(3-2v)}{u^{2}}\sum\limits_{s=1}^{\infty}\frac{(u s+v-1)p^{s}}{(u s+v-1)!}\\
			
			&&+\frac{\eta(1-v)^{2}}{u ^{2}}\sum\limits_{s=1}^{\infty}\frac{p^{s}}{\Gamma(u s+v)}+\frac{(1+\eta\tau|\delta|)}{u}\sum\limits_{s=1}^{\infty}\frac{(u s+v-1)p^{s}}{(u s+v-1)!}\\
			
			&&+\frac{(1+\eta\tau|\delta|)(1-v)}{u}\sum\limits_{s=1}^{\infty}\frac{p^{s}}{\Gamma(u s+v)}+ \tau |\delta|\sum\limits_{s=1}^{\infty}\frac{p^{s}}{\Gamma(u s+v)} \Big]\\

			&=&\frac{1}{E_{u,v}(p)} \Big[ \frac{\eta}{u^{2}}\sum\limits_{s=1}^{\infty}\frac{p^{s}}{(u s+v-3)!}+\frac{\eta(3-2v)}{u^{2}}\sum\limits_{s=1}^{\infty}\frac{p^{s}}{(u s+v-2)!}\\
			
			&&+\frac{\eta(1-v)^{2}}{u ^{2}}\sum\limits_{s=1}^{\infty}\frac{p^{s}}{\Gamma(u s+v)}+\frac{(1+\eta\tau|\delta|)}{u}\sum\limits_{s=1}^{\infty}\frac{p^{s}}{(u s+v-2)!}\\
			
			&&+\frac{(1+\eta\tau|\delta|)(1-v)}{u}\sum\limits_{s=1}^{\infty}\frac{p^{s}}{\Gamma(u s+v)}+ \tau |\delta|\sum\limits_{s=1}^{\infty}\frac{p^{s}}{\Gamma(u s+v)} \Big]\\
			
			&=&\frac{1}{E_{u,v}(p)} \Big[ \frac{\eta}{u^{2}}\sum\limits_{s=1}^{\infty}\frac{p^{s}}{\Gamma(u s+v-2)}+\Big(\frac{\eta(3-2v)}{u^{2}}+\frac{(1+\eta\tau|\delta|)}{u}\Big)\sum\limits_{s=1}^{\infty}\frac{p^{s}}{\Gamma(u s+v-1)}\\

			&&+\Big(\frac{\eta(1-v)^{2}}{u^{2}}+\frac{(1+\eta\tau|\delta|)(1-v)}{u}+\tau|\delta| \Big)\sum\limits_{s=1}^{\infty}\frac{p^{s}}{\Gamma(u s+v)}\Big]\\

			&=&\frac{1}{E_{u,v}(p)}\Big[\frac{\eta}{u^{2}} \Big(E_{u,v-2}(p)-\frac{1}{\Gamma(v-2)}\Big)+\Big(\frac{\eta(3-2v)}{u^{2}}+\frac{(1+\eta\tau|\delta|)}{u}\Big) \Big(E_{u,v-1}(p)-\frac{1}{\Gamma(v-1)} \Big)\\
			
			&&+\Big(\frac{\eta(1-v)^{2}}{u^{2}}+\frac{(1+\eta\tau|\delta|)(1-v)}{u}+\tau|\delta|\Big) \Big(E_{u,v}(p)-\frac{1}{\Gamma(v)}\Big) \Big]

		\end{array}
	\end{equation*}
	Therefore the inequality (2.2) holds if and only if
	\begin{equation*}
		\begin{array}{rcl}
			\frac{1}{E_{u,v}(p)}\Big[\frac{\eta}{u^{2}} \Big(E_{u,v-2}(p)-\frac{1}{\Gamma(v-2)}\Big)+\Big(\frac{\eta(3-2v)}{u^{2}}+\frac{(1+\eta\tau|\delta|)}{u}\Big) \Big(E_{u,v-1}(p)-\frac{1}{\Gamma(v-1)} \Big)\\
			
			+\Big(\frac{\eta(1-v)^{2}}{u^{2}}+\frac{(1+\eta\tau|\delta|)(1-v)}{u}+\tau|\delta|\Big) \Big(E_{u,v}(p)-\frac{1}{\Gamma(v)}\Big) \Big] \leq \tau|\delta|.
		\end{array}
	\end{equation*}
	This completes the proof.
\end{proof}

\begin{theorem}Let $p,u \geq 1, v>2.$  The function $\mathcal{P}(p,u,v)(z) \in\in \mathcal{R}(\delta,\eta,v)$ if and only if
	\begin{equation}
		\begin{array}{rcl}
			\frac{1}{u^{2}E_{u,v}(p)} \Big[\frac{\eta}{u^{2}}\Big(E_{u,v-2}(p)-\frac{1}{\Gamma(v-2)}\Big)+\Big(\frac{\eta(3-2v)+u(\eta+1)}{u^{2}} \Big) \Big(E_{u,v-1}(p)-\frac{1}{\Gamma(v-1)}\Big)\\
			
			+\Big(\frac{\eta(1-v)^{2}}{u ^{2}}+\frac{(\eta +1)(1-v)}{u}+1 \Big) \Big(E_{u,v}(p)-\frac{1}{\Gamma v}\Big) \Big] \leq \tau |\delta|.
		\end{array}
	\end{equation}
\end{theorem}
\begin{proof} We establish that in view of lemma 1.3
	\begin{equation*}
		\sum\limits_{s=2}^{\infty}s[\eta(s-1)+1]
		\frac{p^{s-1}}{E_{u,v}(p)\Gamma(u(s-1)+v)}\leq\tau|\delta|
	\end{equation*}
	Now,
	\begin{equation*}
		\begin{array}{rcl}
			&&\sum\limits_{s=2}^{\infty}s[\eta(s-1)+1]
			\frac{p^{s-1}}{E_{u,v}(p)\Gamma(u(s-1)+v)}\\
			
			&=&\sum\limits_{s=1}^{\infty}\left[(s+1)\eta s+(s+1)\right]\frac{p^{s}}{E_{u,v}(p)\Gamma(u(s)+v)}\\

			&=&\sum\limits_{s=1}^{\infty}\left[\eta s^{2}+s(\eta+1)+1\right]\frac{p^{s}}{E_{u,v}(p)\Gamma(u s+v)}\\

			&=&\frac{1}{E_{u,v}(p)}\sum\limits_{s=1}^{\infty}\Big[\frac{\eta}{u ^{2}}\Big[(u s+v-1)(u s+v-2)+(3-2v)(u s+v-1)+(1-v)^{2}\Big]\\
			
			&&+\frac{(1+\eta)}{u}\Big[(u s+v-1)+(1-v)\Big]+ 1 \Big]\frac{p^{s}}{\Gamma (u s+v)}\\
			
			&=&\frac{1}{E_{u,v}(p)} \Big[ \frac{\eta}{u^{2}}\sum\limits_{s=1}^{\infty}\frac{(u s+v-1)(u s+v-2)p^{s}}{(u s+v-1)!}+\frac{\eta(3-2v)}{u^{2}}\sum\limits_{s=1}^{\infty}\frac{(u s+v-1)p^{s}}{(u s+v-1)!}\\
			
			&&+\frac{\eta(1-v)^{2}}{u ^{2}}\sum\limits_{s=1}^{\infty}\frac{p^{s}}{\Gamma(u s+v)}+\frac{(1+\eta)}{u}\sum\limits_{s=1}^{\infty}\frac{(u s+v-1)p^{s}}{(u s+v-1)!}\\
			
			&&+\frac{(1+\eta)(1-v)}{u}\sum\limits_{s=1}^{\infty}\frac{p^{s}}{\Gamma(u s+v)}+ \sum\limits_{s=1}^{\infty}\frac{p^{s}}{\Gamma(u s+v)} \Big] \\
			
			&=&\frac{1}{E_{u,v}(p)} \Big[ \frac{\eta}{u^{2}}\sum\limits_{s=1}^{\infty}\frac{p^{s}}{(u s+v-3)!}+\frac{\eta(3-2v)}{u^{2}}\sum\limits_{s=1}^{\infty}\frac{p^{s}}{(u s+v-2)!}\\
			
			&&+\frac{\eta(1-v)^{2}}{u ^{2}}\sum\limits_{s=1}^{\infty}\frac{p^{s}}{\Gamma(u s+v)}+\frac{(1+\eta)}{u}\sum\limits_{s=1}^{\infty}\frac{p^{s}}{(u s+v-2)!}\\
			
			&&+\frac{(1+\eta)(1-v)}{u}\sum\limits_{s=1}^{\infty}\frac{p^{s}}{\Gamma(u s+v)}+ \sum\limits_{s=1}^{\infty}\frac{p^{s}}{\Gamma(u s+v)} \Big]\\
			
			&=&\frac{1}{E_{u,v}(p)} \Big[ \frac{\eta}{u^{2}}\sum\limits_{s=1}^{\infty}\frac{p^{s}}{\Gamma(u s+v-2)}+\Big(\frac{\eta(3-2v)}{u^{2}}+\frac{(1+\eta)}{u}\Big)\sum\limits_{s=1}^{\infty}\frac{p^{s}}{\Gamma(u s+v-1)}\\

			&&+\Big(\frac{\eta(1-v)^{2}}{u^{2}}+\frac{(1+\eta)(1-v)}{u}+1 \Big)\sum\limits_{s=1}^{\infty}\frac{p^{s}}{\Gamma(u s+v)}\Big]\\
			
			&=&\frac{1}{u^{2}E_{u,v}(p)} \Big[\frac{\eta}{u^{2}}\Big(E_{u,v-2}(p)-\frac{1}{\Gamma(v-2)}\Big)+\Big(\frac{\eta(3-2v)+u(\eta+1)}{u^{2}} \Big) \Big(E_{u,v-1}(p)-\frac{1}{\Gamma(v-1)}\Big)\\
			
			&&+\Big(\frac{\eta(1-v)^{2}}{u ^{2}}+\frac{(\eta +1)(1-v)}{u}+1 \Big) \Big(E_{u,v}(p)-\frac{1}{\Gamma v}\Big) \Big].
		\end{array}
	\end{equation*}
	This above expression is bounded above by $\tau|\delta|$ if and only if equation (2.3) is satisfied.
\end{proof}

\begin{theorem} A sufficient condition for the function $\mathcal{F}(p,u,v) \in \mathcal{s}(\eta,\delta)$ is
	\begin{equation}
		\begin{array}{rcl}
			&&\frac{2}{E_{u,v}(p)}\Big[\frac{\eta}{u^{2}} \Big(E_{u,v-2}(p)-\frac{1}{\Gamma(v-2)}\Big)+\Big(\frac{\eta(3-2v)}{u^{2}}+\frac{(1+\eta|\delta|)}{u}\Big) \Big(E_{u,v-1}(p)-\frac{1}{\Gamma(v-1)} \Big)\\
			
			&+&\Big(\frac{\eta(1-v)^{2}}{u^{2}}+\frac{(1+\eta|\delta|)(1-v)}{u}+|\delta|\Big) \Big(E_{u,v}(p)-\frac{1}{\Gamma(v)}\Big) \Big]\leq 2|\delta|.
		\end{array}
	\end{equation}
	Using the Lemma 1.5 and applying a similar procedure given to Theorem 2.1, we obtain the above result. Hence the proof is omitted.
\end{theorem}

\begin{theorem} A sufficient condition for the function $\mathcal{F}(p,u,v) \in \mathcal{R}(\eta,\delta)$ is
	\begin{equation}
		\begin{array}{rcl}
			&&\frac{1}{u^{2}E_{u,v}(p)} \Big[\frac{\eta}{u^{2}}\Big(E_{u,v-2}(p)-\frac{1}{\Gamma(v-2)}\Big)+\Big(\frac{\eta(3-2v)+u(\eta+1)}{u^{2}} \Big) \Big(E_{u,v-1}(p)-\frac{1}{\Gamma(v-1)}\Big)\\
			
			&+&\Big(\frac{\eta(1-v)^{2}}{u ^{2}}+\frac{(\eta +1)(1-v)}{u}+1 \Big) \Big(E_{u,v}(p)-\frac{1}{\Gamma v}\Big) \Big] \leq |\delta|.
		\end{array}
	\end{equation}
	Using the Lemma 1.6 and applying a similar procedure given to Theorem 2.2, we obtain the above result. Hence the proof is omitted.
\end{theorem}

\section{\textbf{Integral operators}}
In this section, we define integral operators as follows:
\begin{equation*}
	\mathcal{G}(p,u,v)(z) =\int_{0}^{z}\frac{\mathcal{P}(p,u,v)(t)}{t}dt\  \ \text{ and }\  \ \mathcal{H}(p,u,v)(z) =\int_{0}^{z}\frac{\mathcal{F}(p,u,v)(t)}{t}dt.
\end{equation*}
\begin{theorem} Let $p,u \geq 1, v>1$, then $\mathcal{G}(p,u,v)(z) \in \mathcal{R}(\delta,\eta,v)$ if and only if
	\begin{equation}
		\frac{1}{u E_{u,v}(p)}\Big[\eta\Big(E_{u,v-1}(p)-\frac{1}{\Gamma(v-1)}\Big)+\Big((1-v)\eta+u\Big)\Big(E_{u,v}(p)-\frac{1}{\Gamma v}\Big)\Big]\leq \tau|\delta|
	\end{equation}
\end{theorem}
\begin{proof} We establish that in view of lemma 1.3
	\begin{equation*}
		\sum\limits_{s=2}^{\infty}[\eta(s-1)+1]
		\frac{p^{s-1}}{E_{u,v}(p)\Gamma(u(s-1)+v)}\leq v|\delta|
	\end{equation*}
	Now,
	\begin{equation*}
		\begin{array}{rcl}
			&&\sum\limits_{s=2}^{\infty}[\eta(s-1)+1]
			\frac{p^{s-1}}{E_{u,v}(p)\Gamma(u(s-1)+v)}\\
			&=&\sum\limits_{s=1}^{\infty}[\eta s+1]
			\frac{p^{s}}{E_{u,v}(p)\Gamma(u s+v)}\\
			&=&\sum\limits_{s=1}^{\infty}\Big[\frac{\eta}{u}\Big((u s+v-1)+(1-v)\Big)+1\Big]
			\frac{p^{s}}{E_{u,v}(p)\Gamma(u s+v)}\\
			&=&\frac{1}{E_{u,v}(p)}\Big[\frac{\eta}{u}\sum\limits_{s=1}^{\infty}\frac{(u s+v-1)p^{s}}{(u s+v-1)!}+\frac{(1-v)\eta}{u}\sum\limits_{s=1}^{\infty}\frac{p^{s}}{\Gamma(u s+v)}+\sum\limits_{s=1}^{\infty}\frac{p^{s}}{\Gamma(u s+v)}\Big]\\
			&=&\frac{1}{E_{u,v}(p)}\Big[\frac{\eta}{u}\sum\limits_{s=1}^{\infty}\frac{p^{s}}{\Gamma(u s+v-1)}+\Big(\frac{(1-v)\eta}{u}+1\Big)\sum\limits_{s=1}^{\infty}\frac{p^{s}}{\Gamma(u s+v)}\Big]\\
			&=&\frac{1}{E_{u,v}(p)}\Big[\frac{\eta}{u}\Big(E_{u,v-1}(p)-\frac{1}{\Gamma(v-1)}\Big)+\Big(\frac{(1-v)\eta}{u}+1\Big)\Big(E_{u,v}(p)-\frac{1}{\Gamma v}\Big)\Big].
		\end{array}
	\end{equation*}
	The last expression is bounded above by $\tau |\delta|$ if and only if inequality (3.1) is satisfied.
\end{proof}
\begin{theorem} Let $p,u \geq 1, v>1$, then $\mathcal{H}(p,u,v) \in \mathcal{R}(\eta,\delta)$ if
	\begin{equation}
		\frac{1}{u E_{u,v}(p)}\Big[\Big((1-v)\eta+u\Big)\Big(E_{u,v}(p)-\frac{1}{\Gamma v}\Big)+\eta\Big(E_{u,v-1}(p)-\frac{1}{\Gamma(v-1)}\Big)\Big]\leq |\delta|
	\end{equation}
	We excluded Theorem 3.2's proof because it is very similar to Theorem 3.1's proof.
\end{theorem}

\section{\textbf{Special cases}}
The following corollaries are obtained by specialising different values of the parameter $\eta$ and $\delta$ in Theorems 2.3.\\
For $\eta=0$, we obtain the following result for starlike functions of $\delta$:\\
\textbf{Corrollary 4.1.} We have $\mathcal{F}(p,u,v)(z) \in \mathcal{S}(\delta)$ if
\begin{equation}
	\frac{1}{E_{u,v}(p)}\Big[\frac{1}{u} \Big(E_{u,v-1}(p)-\frac{1}{\Gamma(v-1)} \Big)+\Big(\frac{(1-v)}{u}+|\delta|\Big) \Big(E_{u,v}(p)-\frac{1}{\Gamma v}\Big) \Big]\leq |\delta|.
\end{equation}
For $\eta=1$, we obtain the following result for convex functions of $\delta$:\\
\textbf{Corollary 4.2.}   We have $\mathcal{F}(p,u,v)(z) \in \mathcal{C}(\delta)$ if
\begin{equation}
	\begin{array}{rcl}
		\frac{1}{E_{u,v}(p)}\Big[\frac{1}{u^{2}} \Big(E_{u,v-2}(p)-\frac{1}{\Gamma(v-2)}\Big)+\Big(\frac{(3-2v)}{u^{2}}+\frac{(1+|\delta|)}{u}\Big) \Big(E_{u,v-1}(p)-\frac{1}{\Gamma(v-1)} \Big)\\
		
		+\Big(\frac{(1-v)^{2}}{u^{2}}+\frac{(1+|\delta|)(1-v)}{u}+|\delta|\Big) \Big(E_{u,v}(p)-\frac{1}{\Gamma(v)}\Big) \Big]\leq |\delta|.
	\end{array}
\end{equation}
For $\eta=1,\delta=1$, we obtain the following result for the class of convex functions:
\\
\textbf{Corollary 4.3.} We have $\mathcal{F}(p,u,v)(z) \in \mathcal{C}$ if
\begin{equation}
	\begin{array}{rcl}
		\frac{1}{E_{u,v}(p)}\Big[\frac{1}{u^{2}} \Big(E_{u,v-2}(p)-\frac{1}{\Gamma(v-2)}\Big)+\Big(\frac{(3-2v)}{u^{2}}+\frac{2}{u}\Big) \Big(E_{u,v-1}(p)-\frac{1}{\Gamma(v-1)} \Big)\\
		
		+\Big(\frac{(1-v)^{2}}{u^{2}}+\frac{2(1-v)}{u}+1\Big) \Big(E_{u,v}(p)-\frac{1}{\Gamma(v)}\Big) \Big]\leq 1.
	\end{array}
\end{equation}
For $\eta=0,\delta=1$, we obtain the following result for the class of starlike functions:
\\
\textbf{Corollary 4.4.}  We have $\mathcal{F}(p,u,v)(z) \in \mathcal{S}^{*}$ if
\begin{equation}
	\frac{1}{E_{u,v}(p)}\Big[\frac{1}{u} \Big(E_{u,v-1}(p)-\frac{1}{\Gamma(v-1)} \Big)+\Big(\frac{(1-v)}{u}+1\Big) \Big(E_{u,v}(p)-\frac{1}{\Gamma v}\Big) \Big]\leq 1.
\end{equation}
For $\eta=0,\delta=1-\zeta$, we have the following result for the class of starlike functions of order $\zeta$:
\\
\textbf{Corollary 4.5.}   We have $\mathcal{F}(p,u,v)(z) \in \mathcal{S}^{*}(\zeta)$ if
\begin{equation}
	\begin{array}{rcl}
		\frac{1}{E_{u,v}(p)}\Big[\frac{1}{u} \Big(E_{u,v-1}(p)-\frac{1}{\Gamma(v-1)} \Big)+\Big(\frac{(1-v)}{u}\Big) \Big(E_{u,v}(p)-\frac{1}{\Gamma v}\Big) \Big]\\
		
		\leq |1-\zeta|\Big[1-\Big(\frac{1}{E_{u,v}(p)}\Big(E_{u,v}(p)-\frac{1}{\Gamma v}\Big)\Big)\Big].
	\end{array}
\end{equation}
For $\eta=0,\delta=e^{-i\theta}cos\theta$, we have the following result for the class of spirallike functions:
\\
\textbf{Corollary 4.6.} We have $\mathcal{F}(p,u,v)(z) \in \mathcal{S}^{\theta}$ if
\begin{equation}
	\begin{array}{rcl}
		\frac{1}{E_{u,v}(p)}\Big[\frac{1}{u} \Big(E_{u,v-1}(p)-\frac{1}{\Gamma(v-1)} \Big)+\Big(\frac{(1-v)}{u}\Big) \Big(E_{u,v}(p)-\frac{1}{\Gamma v}\Big) \Big]\\
		
		\leq|\cos\theta|\Big[1-\Big(\frac{1}{E_{u,v}(p)}\Big(E_{u,v}(p)-\frac{1}{\Gamma v}\Big)\Big)\Big].
	\end{array}
\end{equation}
For $\eta=0,\delta=(1-\zeta)e^{-i\theta}cos\theta$, we have the following result for the class of spirallike functions of order $\zeta$:
\\
\textbf{Corollary 4.7.} We have $\mathcal{F}(p,u,v)(z) \in \mathcal{S}^{\theta}(1-\zeta)$ if
\begin{equation}
	\begin{array}{rcl}
		\frac{1}{E_{u,v}(p)}\Big[\frac{1}{u} \Big(E_{u,v-1}(p)-\frac{1}{\Gamma(v-1)} \Big)+\Big(\frac{(1-v)}{u}\Big) \Big(E_{u,v}(p)-\frac{1}{\Gamma v}\Big) \Big]\\
		
		\leq(1-\zeta)|\cos\theta|\Big[1-\Big(\frac{1}{E_{u,v}(p)}\Big(E_{u,v}(p)-\frac{1}{\Gamma v}\Big)\Big)\Big].
	\end{array}
\end{equation}
For $\eta=0$ in the Theorem 2.4, we obtain following result:\\
\textbf{Corollary 4.8.} We have $\mathcal{F}(p,u,v)(z) \in \mathcal{R}(\eta)$ if
\begin{equation}
	\begin{array}{rcl}
		\frac{1}{u^{2}E_{u,v}(p)} \Big[\Big(\frac{\eta(3-2v)+u(\eta+1)}{u^{2}} \Big) \Big(E_{u,v-1}(p)-\frac{1}{\Gamma(v-1)}\Big)\\
		
		+\Big(\frac{\eta(1-v)^{2}}{u ^{2}}+\frac{(\eta +1)(1-v)}{u}+1 \Big) \Big(E_{u,v}(p)-\frac{1}{\Gamma v}\Big) \Big] \leq |\delta|.
	\end{array}
\end{equation}

\section{\textbf{Conclusion}} This paper deals with the different subclasses of complex order univalent function associted with Mittag–Leffler-type Poisson distribution series. During our investigation, we have extensively studied a particular subclasses $\mathcal{S}(\delta,\eta,\tau)$ and $\mathcal{R}(\delta,\eta,\tau)$. Our analysis involved deriving coefficient inequality, which helped us gain insights into the behavior and structure of functions within this subclass. Moreover, we have established distortion bounds that provide valuable information on the preservation of geometric properties under mapping. Here we have obtain the results of integral theorem and many interesting special cases.

\section*{\bf Declaration}
\begin{itemize}
    \item Availability of data and materials: The data is provided on request to the authors.
    \item Conflicts of interest: The authors declare that they have no conflicts of interest and all agree to publish this paper under academic ethics.

    \item Author’s contribution: All the authors equally contributed towards this work.
\end{itemize}
%%%%%%%%%%%%%%%%%%%%%%%%%%%%%%%%%%%%%%%%%%%%%%%%%%%%%%%%%%%%%%%%%%%%%****************************************

%\section*{Conclusion}\label{sectconclusion}

%%%%%%%%%%%%%%%%%%%%%%%%%%%%%%%%%%%%%%%%%%%%%%%%%%%%%%%%%%%%%%%%%%%%%****************************************

%%%%%%%%%%%%%%%%%%%%%%%%%%%%%%%%%%%%%%%%%%%%%%%%%%%%%%%%%%%%%%%%%%%%%****************************************


\begin{thebibliography}{99}

\bibitem{Alessa}N. Alessa\ et al., A new subclass of analytic functions related to Mittag-Leffler type Poisson distribution series, J. Funct. Spaces {\bf 2021}, Art. ID 6618163, 7 pp.

\bibitem{Altintas 1999}O. Alt\i nta\c{s}\ and\ \"{O}zkan, Starlike, convex and close-to-convex functions of complex order, Hacet. Bull. Nat. Sci. Eng. Ser. B {\bf 28} (1999), 37--46.

\bibitem{Altintas 2000}O. Altinta\c{s}, \H{O}zkan\ and\ H. M. Srivastava, Neighborhoods of a class of analytic functions with negative coefficients, Appl. Math. Lett. {\bf 13} (2000), no.~3, 63--67.



\bibitem{Aouf 2013} M. K. Aouf, Some subordinations results for certain subclasses of starlike and convex functions of complex order, Acta Univ. Apulensis Math. Inform. No. 35 (2013), 101--110.

\bibitem{Bajpai} D.Bajpai: A study of univalent functions associated with distortion series and q- calculus, M.Phil.,dissertation, CSJM Univerity, Kanpur, India 2016.

\bibitem{D.Bansal}D. Bansal\ and\ J. K. Prajapat, Certain geometric properties of the Mittag-Leffler functions, Complex Var. Elliptic Equ. {\bf 61} (2016), no.~3, 338--350.


\bibitem{Abdul Halim 1999}S. Abdul Halim, On a class of functions of complex order, Tamkang J. Math. {\bf 30} (1999), no.~2, 147--153.

\bibitem{M.G.Khan}M. G. Khan\ et al., Applications of Mittag-Leffer Type Poisson Distribution to a Subclass of Analytic Functions Involving Conic-Type Regions, J. Funct. Spaces {\bf 2021}, Art. ID 4343163, 9 pp.

\bibitem{Mari 2023}
 K. Marimuthu, J. Uma, and T.  Teodor Bulboacă,  Coefficient estimates for starlike and convex functions associated with cosine function, Hacet. J. Math. Stat. {\bf 2023}, 52(3), 596–618.

\bibitem{G.M 2020}G. Murugusundaramoorthy,B.A Frasin, T. Al-Hawary. Uniformly convex spiral functions and uniformly spirallike function associated with Pascal distribution series. arXiv 2020; arXiv:2001.07517.

\bibitem{Nasr 1985}M. A. Nasr\ and\ M. K. Aouf, Starlike function of complex order, J. Natur. Sci. Math. {\bf 25} (1985), no.~1, 1--12.

\bibitem{Owa 1988}S. Owa, Notes on starlike, convex, and close-to-convex functions of complex order, in {\it Univalent functions, fractional calculus, and their applications (Koriyama, 1988)}, 199--218, Ellis Horwood Ser. Math. Appl, Horwood, Chichester.

\bibitem{Porwal 2017}S. Porwal\ and\ K. K. Dixit, On Mittag-Leffler type Poisson distribution, Afr. Mat. {\bf 28} (2017), no.~1-2, 29--34.

\bibitem{Porwal 2020}S. Porwal, N. Magesh\ and\ C. Abirami, Certain subclasses of analytic functions associated with Mittag-Leffler-type Poisson distribution series, Bol. Soc. Mat. Mex. (3) {\bf 26} (2020), no.~3, 1035--1043.



\bibitem{Raducanu}D. R\u{a}ducanu, Third-order differential subordinations for analytic functions associated with generalized Mittag-Leffler functions, Mediterr. J. Math. {\bf 14} (2017), no.~4, Paper No. 167, 18 pp.

\bibitem{V 2021}L.Vanitha, C.Ramachandran and T.Bulboaca, Certain subclasses of spirallike univalent functions related to poisson distributions serties Turk J Math 45(2021), 1449-1458.

\bibitem{Yalcin 2024}
 S. Yalçın, K. Marimuthu, and J. Uma,  An application of generalized distribution series on certain subclasses of univalent functions, Bol. Soc. Mat. Mex.{\bf 2024} Vol.30, Art.no 3.

\bibitem{W 1971}P. Wiatrowski, The coefficients of a certain family of holomorphic functions, Zeszyty Nauk. Uniw. \L \'{o}dz. Nauki Mat. Przyrod. Ser. II No. 39 Mat. (1971), 75--85.


\end{thebibliography}
\end{document}